\documentclass[a4paper]{article}
\usepackage{amsmath, amsthm, amssymb}
\usepackage[latin1]{inputenc}
\usepackage[T1]{fontenc}
\usepackage[english]{babel}     
\usepackage{color}
\usepackage{tikz}                      
\usepackage{pgf}
\usepackage{xkeyval}

\usepackage{multirow}           

\usepackage[dvips]{epsfig}
\usepackage{psfrag}
\usepackage{algorithm}
\usepackage{algorithmic}

\usepackage{listings}           
\usepackage{color}

\definecolor{dkgreen}{rgb}{0,0.6,0}
\definecolor{gray}{rgb}{0.5,0.5,0.5}
\definecolor{mauve}{rgb}{0.58,0,0.82}

\newtheorem{Corollary}{Corollary}[section]
\newtheorem{Lemma}{Lemma}[section]
\newtheorem{Theorem}{Theorem}[section]
\newtheorem{Remark}{Remark}[section]

\newtheorem{Example}{Example}[section]
\numberwithin{equation}{section}
\def\dvg{{\rm div}}

\newcommand\be{\begin{eqnarray*}}
\newcommand\ee{\end{eqnarray*}}
\newcommand\ben{\begin{eqnarray}}
\newcommand\een{\end{eqnarray}}
\newcommand{\comment}[1]{}

\def\dvg{{\rm div}}

\def\IntO{\int\limits_\Omega}

\def\blow{\underline{c}}

\def\dx{\, \textrm{d}\mathbf{x}}

\def\maj{\overline{\mathfrak{M}}}

\def\H1o{{\stackrel{\circ}{H^1}}(\Omega)}   
\def\Hdiv{H(\dvg,\Omega)}                   

\title{Functional a posteriori error estimate for a nonsymmetric stationary diffusion problem}
\author{O. Mali}

\begin{document}

\maketitle

\begin{abstract}
In this paper, a posteriori error estimates of functional type for
a stationary diffusion problem with nonsymmetric coefficients are
derived. The estimate is guaranteed and does not depend on any
particular numerical method. An algorithm for the global
minimization of the error estimate with respect to the flux over
some finite dimensional subspace is presented. In numerical tests,
global minimization is done over the subspace generated by
Raviart-Thomas elements. The improvement of the error bound due to
the $p$-refinement of these spaces is investigated.
\end{abstract}

\section{Introduction}

In this paper, we derive a posteriori error estimates of the
functional type for a class of elliptic problems with nonsymmetric
coefficients. Since mid 90's (see \cite{Repin1997}), estimates of
this type has been derived for a wide range of problems (see,
e.g.,  monographs
\cite{NeittaanmakiRepin2004,Repin2008,MaliNeittaanmakiRepin2014}
and references there in). However, the case of a stationary
diffusion problem, where coefficients are not symmetric has not
been studied before. Problems of this type are not very typical
among other elliptic equations but they arise in certain models
(see, e.g., \cite{Denisenko1989,Denisenko1997}). It is shown that
the derived estimate has the standard properties of a deviation
estimate for a linear problem, i.e., it is guaranteed and
computable. The derivation of the estimate is based on the method
of integral identities and a special case of
Cauchy-Schwartz-Bunyakovsky inequality.


Consider the Poisson problem,
\begin{align}
\label{eq:diff:strong1}
- \dvg \mathbf{A} \nabla u = f & \quad \textrm{in} \; \Omega \subset \mathbb{R}^d\\
\label{eq:diff:strong2}
                 u = 0 & \quad \textrm{on} \; \Gamma ,
\end{align}
where $\Omega$ a is simply connected domain with a Lipschitz-continuous boundary,
$f \in L^2(\Omega)$, and $\mathbf{A} \in L_\infty(\Omega,\mathbb{R}^{d \times d})$
is strictly positive definite, bounded, and has a bounded inverse
$\mathbf{A}^{-1} \in \mathbb{R}^{d \times d}$ in $\Omega$. Moreover,
$\mathbf{A}$ is positive definite, i.e., there exists constant $\blow >0$ such that
\begin{equation} \label{eq:blow}
(\mathbf{A} \boldsymbol{\xi}, \boldsymbol{\xi})_{\mathbb{R}^d}
\geq \blow
\| \boldsymbol{\xi} \|_{\mathbb{R}^d}^2,
\quad \forall \boldsymbol{\xi} \in \mathbb{R}^d, \; \textrm{a.e. in } \Omega .
\end{equation}
The generalized solution
$u \in H_0^1(\Omega)$ satisfies the integral identity,
\begin{equation} \label{eq:diff:weak}
( \mathbf{A} \nabla u , \nabla w )_{L^2(\Omega,\mathbb{R}^d)}
=
( f , w )_{L^2(\Omega)}, \quad
\forall w \in H_0^1(\Omega) .
\end{equation}


\section{Error majorant}


For symmetric problems with
$\mathbf{A} \in L_\infty(\Omega,\mathbb{R}^{d \times d}_{\rm sym})$
the respective guaranteed upper bounds (error majorants) have been
presented in
\cite{NeittaanmakiRepin2004,Repin2008,MaliNeittaanmakiRepin2014}
and other publications cited therein. It has the form,
\begin{equation*} \label{eq:diff:maj}
\maj(v,\mathbf{y}) :=
(\mathbf{A}\nabla v - \mathbf{y},
\nabla v - \mathbf{A}^{-1} \mathbf{y})^{1/2}_{L^2(\Omega,\mathbb{R}^d)}
+ \frac{C_F}{\sqrt{\blow}} \| \dvg \mathbf{y} + f \|_{L^2(\Omega)} ,
\end{equation*}
where $v \in H_0^1(\Omega)$, $\mathbf{y} \in \Hdiv$, and $C_F$ is the constant in
Friedrichs inequality,
\begin{equation} \label{eq:Fri}
\| w \|_{L^2(\Omega)} \leq C_F \| \nabla w \|_{L^2(\Omega,\mathbb{R}^d)},
\quad \forall w \in H_0^1(\Omega).
\end{equation}

A special case of the Cauchy-Schwartz-Bunyakovsky inequality presented below is
required to obtain an analogous error estimate in the nonsymmetric case.
\begin{Lemma} \label{le:nonsym}
Let $\mathcal U$ be a Hilbert space which field is real numbers,
$A:\mathcal U \rightarrow \mathcal U$ is continuous, bounded,
strictly positive definite, and has a continuous inverse
$A^{-1}$. Moreover,
\begin{equation*} 
B := ( {\rm Id} + A^T A^{-1} )^{-1}
\end{equation*}
is continuous and bounded. Then,
\begin{equation} \label{eq:nonsym}
(y,q)_{\mathcal U} \leq
2
(A y,y)_{\mathcal U}^{1/2}
(A^{-1} B q, B q)_{\mathcal U}^{1/2} , \quad \forall y,q \in \mathcal U .
\end{equation}
\end{Lemma}
\begin{proof}
Since $A$ is strictly positive definite,
\begin{multline*}
0 \leq ( A (y-\gamma A^{-1} q),
y-\gamma A^{-1} q)_{\mathcal U} \\
=
(A y,y)_{\mathcal U}
- \gamma (y, ({\rm Id}+ A^T A^{-1}) q )_{\mathcal U}
+
\gamma^2 ( A^{-1} q , q )_{\mathcal U} .
\end{multline*}
Selecting (assume $y \neq 0$ and $q \neq 0$,
otherwise (\ref{eq:nonsym}) holds trivially)
\[
\gamma = \frac{2 (A y,y)_{\mathcal U}}
{(y, ({\rm Id}+ A^T A^{-1}) q )_{\mathcal U}}
\]
yields
\[
(y, ({\rm Id}+ A^T A^{-1}) q )_{\mathcal U}^2
\leq
4
(A y,y)_{\mathcal U}
( A^{-1} q , q )_{\mathcal U},
\]
where setting $q = Bq = ({\rm Id}+ A^T A^{-1})^{-1} q$
leads at (\ref{eq:nonsym}).
\end{proof}

\begin{Theorem}
Let $v \in H^1_0(\Omega)$ and $u$ be the solution of (\ref{eq:diff:weak}), then,
\[
( \mathbf{A} \nabla (u-v), \nabla (u-v) )_{L^2(\Omega,\mathbb{R}^d)}^{1/2}
\leq
\maj(v,\mathbf{y}) , \quad \forall \mathbf{y} \in \Hdiv ,
\]
where
\begin{equation*} 
\maj(v,\mathbf{y}) :=
2 ( \mathbf{A}^{-1} \mathbf{B} (\mathbf{y} - \mathbf{A} \nabla v),
\mathbf{B} (\mathbf{y} - \mathbf{A} \nabla v) )_{L^2(\Omega,\mathbb{R}^d)}^{1/2}
+
\frac{C_F}{\sqrt{\blow}} \| \dvg \mathbf{y} + f \|_{L^2(\Omega)}
\end{equation*}
and
\[
\mathbf{B} := ( \mathbf{I} + \mathbf{A}^{T} \mathbf{A}^{-1} )^{-1}.
\]
The constants $C_{F}$ and $\blow$ are defined in (\ref{eq:Fri}) and
(\ref{eq:blow}), respectively.
\end{Theorem}
\begin{proof}
Subtracting $\mathbf{A}\nabla v$ from both sides of (\ref{eq:diff:weak})
and applying the integration by parts formula
\[
( \mathbf{y} , \nabla w)_{L^2(\Omega,\mathbb{R}^d)}
=
(- \dvg \mathbf{y} , w)_{L^2(\Omega)} , \quad
\forall
\mathbf{y} \in \Hdiv, \;
w \in H^1_0(\Omega)
\]
yields
\[
( \mathbf{A} \nabla (u-v), \nabla w )_{L^2(\Omega,\mathbb{R}^d)}
=
( \mathbf{y} - \mathbf{A} \nabla v, \nabla w)_{L^2(\Omega,\mathbb{R}^d)}
+
(\dvg \mathbf{y} + f,w)_{L^2(\Omega)} .
\]
The first term can be estimated from above by (\ref{eq:nonsym}), where
$\mathcal U := L^2(\Omega,\mathbb{R}^d)$ and $A:=\mathbf{A}$. The
second term is estimated from above by Hölder inequality,
(\ref{eq:Fri}), and (\ref{eq:blow}), which leads at
\begin{multline*}
( \mathbf{A} \nabla (u-v), \nabla w )_{L^2(\Omega,\mathbb{R}^d)} \leq \\
2 ( \mathbf{A}^{-1} \mathbf{B} (\mathbf{y} - \mathbf{A} \nabla v),
\mathbf{B} (\mathbf{y} - \mathbf{A} \nabla v) )_{L^2(\Omega,\mathbb{R}^d)}^{1/2}
(\mathbf{A} \nabla w,\nabla w)_{L^2(\Omega,\mathbb{R}^d)}^{1/2} \\
+
\frac{C_F}{\sqrt{\blow}} \| \dvg \mathbf{y} + f \|_{L^2(\Omega)}
(\mathbf{A} \nabla w,\nabla w)_{L^2(\Omega,\mathbb{R}^d)}^{1/2} .
\end{multline*}
Setting $w=u-v$ leads at (\ref{eq:nonsym:maj}).
\end{proof}
\begin{Remark}
Two parts of the majorant are related to the violations of the duality relation and
the equilibrium condition, respectively. They are denoted by
\begin{align*}
\maj_{\rm Dual} & := ( \mathbf{A}^{-1} \mathbf{B} (\mathbf{y} - \mathbf{A} \nabla v),
\mathbf{B} (\mathbf{y} - \mathbf{A} \nabla v) )_{L^2(\Omega,\mathbb{R}^d)}^{1/2}, \\
\maj_{\rm Equi} & := \| \dvg \mathbf{y} + f \| . \\
\end{align*}
\end{Remark}

\section{Global minimization of the error majorant}

Squaring and applying the Young's inequality yields a quadratic form of the majorant, which is more
suitable for the minimization over $\mathbf{y}$.
\begin{Corollary}
Let $v \in H^1_0(\Omega)$ and $u$ be the solution of (\ref{eq:diff:weak}), then,
\[
( \mathbf{A} \nabla (u-v), \nabla (u-v) )_{L^2(\Omega,\mathbb{R}^d)}
\leq
\maj^2(v,\mathbf{y},\beta) , \quad \forall \mathbf{y} \in \Hdiv , \, \beta > 0 ,
\]
where
\begin{multline} \label{eq:nonsym:maj}
\maj^2(v,\mathbf{y},\beta) :=
4 (1+\beta)( \mathbf{A}^{-1} \mathbf{B} (\mathbf{y} - \mathbf{A} \nabla v),
\mathbf{B} (\mathbf{y} - \mathbf{A} \nabla v) )_{L^2(\Omega,\mathbb{R}^d)} \\
+
\frac{1+\beta}{\beta} \frac{C_F^2}{\blow} \| \dvg \mathbf{y} + f \|_{L^2(\Omega)}^2 .
\end{multline}
\end{Corollary}
\begin{Corollary}
The minimizers
\begin{align*}
\maj^2(v,\mathbf{\hat y},\beta) & = \min\limits_{\mathbf{y} \in \Hdiv}
\maj^2(v,\mathbf{y},\beta) \\
\maj^2(v,\mathbf{y},\hat \beta) & = \min\limits_{\beta > 0}
\maj^2(v,\mathbf{y},\beta)
\end{align*}
satisfy
\begin{multline} \label{eq:hatydef}
\frac{C_F^2}{\blow} (\dvg \mathbf{\hat y} , \dvg \mathbf{q})_{L^2(\Omega)}
+
2 \beta \left(
(\mathbf{A}^{-1} \mathbf{B} \mathbf{q}, \mathbf{B} \mathbf{\hat y})_{L^2(\Omega,\mathbb{R}^d)}
+
(\mathbf{A}^{-1} \mathbf{B} \mathbf{\hat y}, \mathbf{B} \mathbf{q})_{L^2(\Omega,\mathbb{R}^d)}
\right) \\
=
-\frac{C_F^2}{\blow} (f , \dvg \mathbf{q})_{L^2(\Omega)}
+
2 \beta \left(
(\mathbf{A}^{-1} \mathbf{B} \mathbf{q}, \mathbf{B} \mathbf{A} \nabla v)_{L^2(\Omega,\mathbb{R}^d)}
+
(\mathbf{A}^{-1} \mathbf{B} \mathbf{A} \nabla v, \mathbf{B} \mathbf{q})_{L^2(\Omega,\mathbb{R}^d)}
\right), \\
\quad \forall \mathbf{q} \in \Hdiv
\end{multline}
and
\begin{equation} \label{eq:hatbetadef}
\hat \beta = \frac{
\frac{C_F}{\sqrt{\blow}} \| \dvg \mathbf{y} + f \|_{L^2(\Omega)}
}
{
( \mathbf{A}^{-1} \mathbf{B} (\mathbf{y} - \mathbf{A} \nabla v),
\mathbf{B} (\mathbf{y} - \mathbf{A} \nabla v) )_{L^2(\Omega,\mathbb{R}^d)}^{1/2} ,
}
\end{equation}
respectively.
\end{Corollary}
\begin{proof}
The functional $\maj^2(v,\mathbf{y},\beta)$ is quadratic and convex w.r.t. $\mathbf{y}$.
Thus the necessary and sufficient condition for the minimizer $\mathbf{\hat y}$ is
\[
\frac{{\rm d}}{{\rm d} t} \maj^2(v,\mathbf{\hat y}+t \mathbf{q}, \beta)\Big|_{t=0} = 0 ,
\quad \forall \mathbf{q} \in \Hdiv ,
\]
which leads to (\ref{eq:hatydef}). Similarly,
\[
\frac{{\rm d}}{{\rm d} \beta} \maj^2(v,\mathbf{y},\hat \beta) = 0
\]
yields (\ref{eq:hatbetadef}).
\end{proof}
\begin{Remark}
If $\mathbf{A}$ is symmetric, then (\ref{eq:hatydef}) reduces to
\begin{multline*} 
C_F^2 \IntO \dvg \mathbf{\hat y} \dvg \mathbf{q} \dx
+ \beta \IntO \mathbf{A}^{-1} \mathbf{\hat y} \cdot \mathbf{q} \dx \\
=
-C_F^2 \IntO f \dvg \mathbf{q} \dx
+ \beta \IntO \nabla v \cdot \mathbf{q} \dx
\quad \forall \mathbf{q} \in \Hdiv .
\end{multline*}
\end{Remark}

There are many alternatives how to compute the value of the majorant
(see, e.g., \cite[Chap. 3]{MaliNeittaanmakiRepin2014}). Here, the the global
minimization of the majorant over
finite dimensional subspace is presented. The minimization is done iteratively by solving
(\ref{eq:hatydef}) and (\ref{eq:hatbetadef}) subsequently.

Let
$\mathbf{y} = \sum_{j=1}^N c_j \boldsymbol\phi_j$ and
$Q_h := {\rm span}(\boldsymbol\phi_1,\dots,\boldsymbol\phi_N) \subset \Hdiv$, i.e.,
$\boldsymbol\phi_j$ ($j \in \{1,\dots,N\} $) are the global basis functions.
Then (\ref{eq:hatydef}) leads to a system of linear equations
\begin{equation} \label{eq:ymin2}
\left( \frac{C_F^2}{\sqrt{\blow}} \mathbf{S}  + 2 \beta \mathbf{M} \right) \mathbf{c} =
-\frac{C_F^2}{\sqrt{\blow}} \mathbf{b} + 2 \beta \mathbf{z} ,
\end{equation}
where
\begin{align}
\label{eq:matdef1}
& S_{ij} := ( \dvg \boldsymbol{\phi}_j , \dvg \boldsymbol{\phi}_i )_{L^2(\Omega)}, \\
\label{eq:matdef2}
& M_{ij} :=
(\mathbf{A}^{-1} \mathbf{B} \boldsymbol{\phi}_j, \mathbf{B} \boldsymbol{\phi}_i)_{L^2(\Omega,\mathbb{R}^d)}
+
(\mathbf{A}^{-1} \mathbf{B} \boldsymbol{\phi}_i, \mathbf{B} \boldsymbol{\phi}_j)_{L^2(\Omega,\mathbb{R}^d)}
 , \\
\label{eq:matdef3}
& b_{i} := ( f , \dvg \boldsymbol{\phi}_i )_{L^2(\Omega)} \\
\label{eq:matdef4}
& z_{i} := (\mathbf{A}^{-1} \mathbf{B} \boldsymbol{\phi}_i, \mathbf{B} \mathbf{A} \nabla v)_{L^2(\Omega,\mathbb{R}^d)}
+
(\mathbf{A}^{-1} \mathbf{B} \mathbf{A} \nabla v, \mathbf{B} \boldsymbol{\phi}_i)_{L^2(\Omega,\mathbb{R}^d)} ,
\end{align}
and $\mathbf{c} \in \mathbb{R}^N$ is the (column) vector of unknown coefficients.
The natural choice is to generate $Q_h$ using Raviart-Thomas -elements (see
\cite{RaviartThomas1977}).
The global minimization procedure for $\maj^2$ is described in Algorithm \ref{alg:maj}.
\begin{algorithm}[t!]
\caption{Computation of the majorant for
the problem (\ref{eq:diff:strong1})-(\ref{eq:diff:strong2})}
\label{alg:maj}
\begin{algorithmic}
\STATE {\bf Input:} $v$ \COMMENT{approximate solution},
$\mathbf{A}$, \COMMENT{diffusion coefficient matrix}
$f$, \COMMENT{RHS of the problem},
$C_F$, \COMMENT{Constant in (\ref{eq:Fri})},
$\blow$, \COMMENT{Constant in (\ref{eq:blow})},
$I_{\max}$ \COMMENT{maximum number
of iterations}, $\epsilon$ \COMMENT{stopping criteria for $\maj$}
\STATE{}
\STATE Generate $\mathbf{S}$, $\mathbf{M}$, $\mathbf{b}$, and $\mathbf{z}$ in
(\ref{eq:matdef1})-(\ref{eq:matdef4}).
\STATE Compute norms $\| f \|$ and $\|\nabla v \|$.
\STATE Set $\beta_1 := 1$, $\maj_{k}=\infty$ and $k=0$. \COMMENT{initialize parameters}
\WHILE {$k < I_{\max}$ {\bf and} $\frac{\maj_{k+1} - \maj_{k}}{\maj_{k}} > \epsilon$}
    \STATE $k=k+1$
    \STATE Solve $\mathbf{c}_{k+1}$ from
    $
        \left( C_F^2 \mathbf{S}  + 2 \beta_k \mathbf{M} \right) \mathbf{c}_{k+1} =
        -C_F^2 \mathbf{b} + 2 \beta_k \mathbf{z} .
    $
    \STATE
    $
        \maj^{\rm Equi}_{k+1} = \sqrt{
        \mathbf{c}_{k+1}^T \mathbf{S} \mathbf{c}_{k+1}
        + 2 \mathbf{c}_{k+1}^T \mathbf{b} + \|f\|^2 }
    $
    \STATE
    $
        \maj^{\rm Dual}_{k+1} = \sqrt{
        \mathbf{c}_{k+1}^T \mathbf{M} \mathbf{c}_{k+1}
        - 2 \mathbf{c}_{k+1}^T \mathbf{z} + \|\nabla v \|^2 }
    $
%
    \STATE
    $
        \beta_{k+1} = \frac{C_F \maj^{\rm Equi}_{k+1}}{2 \sqrt{\blow} \maj^{\rm Dual}_{k+1}}
    $
    \STATE
    $
    \maj_{k+1} = 2 \maj^{\rm Dual}_{k+1} + \frac{C_F}{\sqrt{\blow}} \maj^{\rm Equi}_{k+1}
    $
\ENDWHILE
\STATE
$
\mathbf{y} = \sum_{j=1}^N {c_k}_j \boldsymbol{\phi}_j
$
\STATE{}
\STATE {\bf Output:} $\maj_{k+1}$ \COMMENT{Upper bound for the approximation error},
$\mathbf{y}$ \COMMENT{Approximation of the flux}
\end{algorithmic}
\end{algorithm}
\begin{Remark}
Note that in Algorithm \ref{alg:maj}, the global matrices $\mathbf{S}$ and $\mathbf{M}$
have to be assembled only once.
The coefficient matrix in (\ref{eq:ymin2}) is symmetric regardless of the fact that
$\mathbf{A}$ is not.
\end{Remark}

\section{Numerical tests}

Algorithm \ref{alg:maj} is very convenient to implement using any
finite element software, e.g., FEniCS \cite{LoggMardalEtAl2012a}
and FREEFEM++ \cite{Hecht2012}), which allows user to define
problems using weak forms. This is true for all estimates of the
functional type presented in
\cite{NeittaanmakiRepin2004,Repin2008,MaliNeittaanmakiRepin2014}.
The following tests are computed using FEniCS finite element
package. Here, we apply Algorithm \ref{alg:maj} to estimate the
error of a finite element approximation for a test example, where
the exact solution is known.
\begin{Example} \label{ex:1}
Let $\Omega = (0,1)\times(0,1)$, 
$u_g = 0$,
$
\mathbf{A} = \left( \begin{smallmatrix} a & b \\ c & d \end{smallmatrix} \right)
$, \linebreak
$u(x_1,x_2)=\sin(k_1 \pi x_1) \sin(k_2 \pi x_2)$, and
\begin{multline*}
f(x_1,x_2) = \pi^2 \left(
(a+d) k_1^2 \sin(k_1 \pi x_1) \sin(k_2 \pi x_2) \right. \\
\left.
- (b+c) k_1 k_2 \cos(k_1 \pi x_1) \cos(k_2 \pi x_2)
\right) .
\end{multline*}
Select
$
\mathbf{A} = \left( \begin{smallmatrix} 2 & 1 \\ 0 & 3 \end{smallmatrix} \right)
$,
then $\blow = 2$,
$
\mathbf{A}^{-1} =
\tfrac{1}{6}
\left( \begin{smallmatrix}
3 & -1 \\ 0 & 2
\end{smallmatrix} \right)
$
and
$
\mathbf{B} =
\tfrac{1}{23}
\left( \begin{smallmatrix}
11 & 2 \\ -3 & 12
\end{smallmatrix} \right) .
$
\end{Example}

The approximate solution $v \in V_h$ of Example \ref{ex:1}
is computed on a mesh $\mathcal T_h$, using triangular Courant elements of the
order $p_1$. The space $Q_h$ is generated using the Raviart-Thomas elements of order $p_2$
on the same mesh. The amount of global degrees of freedom are denoted
by $N_1=\textrm{dim}(V_h)$ and $N_2=\textrm{dim}(Q_h)$
The efficiency index of the majorant is
\begin{equation} \label{eq:Ieff:def}
I_{\rm eff} := \frac{\maj^2(v,\mathbf{y},\beta)}
{(\mathbf{A}\nabla(u-v),\nabla(u-v))_{L^2(\Omega,\mathbb{R}^d)}}
\end{equation}
The majorant is computed for different meshes with $k_1=1$, $k_2=1$, and $p_1=1$
in Table \ref{tab:ex1}.
\begin{table}
\begin{center}
\caption{Example \ref{ex:1}: $k_1=1$, $k_2=1$, and $p_1=1$}
\label{tab:ex1}
\begin{tabular}{@{}lllllllllll@{}}
\hline
$N_1$ & $p_2$ & $N_2$ & $k$ & $\maj^2(v,\mathbf{y}_k,\beta_k)$
& $\maj^{\rm Dual}_k$ & $\maj^{\rm Equi}_k$ & $I_{\rm eff}$ \\
\hline
441 & 1 & 1240 & 3 &   1.76E+00 &   2.46E-02 &   2.06E+00 & 6.6480 \\
441 & 2 & 4080 & 3 &   3.15E-01 &   1.78E-02 &   2.09E-03 & 1.1858 \\
441 & 3 & 8520 & 4 &   2.68E-01 &   1.78E-02 &   1.07E-06 & 1.0090 \\
1681 & 1 & 4880 & 3 &   8.85E-01 &   6.23E-03 &   5.17E-01 & 6.6452 \\
1681 & 2 & 16160 & 3 &   1.45E-01 &   4.44E-03 &   1.31E-04 & 1.0920 \\
1681 & 3 & 33840 & 4 &   1.33E-01 &   4.44E-03 &   1.68E-08 & 1.0023 \\
6561 & 1 & 19360 & 2 &   4.43E-01 &   1.56E-03 &   1.29E-01 & 6.6445 \\
6561 & 2 & 64320 & 3 &   6.97E-02 &   1.11E-03 &   8.20E-06 & 1.0458 \\
6561 & 3 & 134880 & 3 &   6.66E-02 &   1.11E-03 &   2.62E-10 & 1.0006 \\
14641 & 1 & 43440 & 2 &   2.95E-01 &   6.95E-04 &   5.75E-02 & 6.6443 \\
14641 & 2 & 144480 & 3 &   4.58E-02 &   4.93E-04 &   1.62E-06 & 1.0305 \\
14641 & 3 & 303120 & 3 &   4.44E-02 &   4.93E-04 &   2.30E-11 & 1.0003 \\
40401 & 1 & 120400 & 2 &   1.77E-01 &   2.50E-04 &   2.07E-02 & 6.6443 \\
40401 & 2 & 400800 & 3 &   2.71E-02 &   1.78E-04 &   2.10E-07 & 1.0183 \\
40401 & 3 & 841200 & 3 &   2.67E-02 &   1.78E-04 &   1.07E-12 & 1.0002  \\
\hline
\end{tabular}
\end{center}
\end{table}
The efficiency of the majorant and the number of iterations
(in Algorithm \ref{alg:maj} $\varepsilon = 10^{-6}$) do
not depend on the mesh size. For $p_2=2$ and $p_3$,
$Q_h$ can practically present the exact flux, since the efficiency index is almost one.
Note that in this case $\maj^{\rm Dual}$ is almost the exact error and $\maj^{\rm Equi}$
vanishes. Results of a similar experiment in the case $k_1=2$, $k_2=3$, and $p_1=2$ are
depicted in Table \ref{tab:ex2}.
\begin{table}
\begin{center}
\caption{Example \ref{ex:1}: $k_1=2$, $k_2=3$, and $p_1=2$}
\label{tab:ex2}
\begin{tabular}{@{}lllllllllll@{}}
\hline
$N_1$ & $p_2$ & $N_2$ & $k$ & $\maj^2(v,\mathbf{y}_k,\beta_k)$
& $\maj^{\rm Dual}_k$ & $\maj^{\rm Equi}_k$ & $I_{\rm eff}$ \\
\hline
1681 & 1 & 1240 & 3 &   2.60E+01 &   3.94E-01 &   6.10E+02 & 189.9638 \\
1681 & 2 & 4080 & 3 &   2.15E+00 &   6.05E-03 &   3.92E+00 & 15.6634 \\
1681 & 3 & 8520 & 2 &   2.53E-01 &   4.71E-03 &   1.26E-02 & 1.8496 \\
6561 & 1 & 4880 & 3 &   1.32E+01 &   9.51E-02 &   1.56E+02 & 380.2599 \\
6561 & 2 & 16160 & 3 &   5.41E-01 &   3.89E-04 &   2.49E-01 & 15.6199 \\
6561 & 3 & 33840 & 3 &   4.93E-02 &   3.00E-04 &   1.99E-04 & 1.4258 \\
25921 & 1 & 19360 & 3 &   6.60E+00 &   2.36E-02 &   3.94E+01 & 760.6287 \\
25921 & 2 & 64320 & 2 &   1.35E-01 &   2.45E-05 &   1.56E-02 & 15.6082 \\
25921 & 3 & 134880 & 3 &   1.05E-02 &   1.88E-05 &   3.12E-06 & 1.2139 \\
58081 & 1 & 43440 & 3 &   4.40E+00 &   1.05E-02 &   1.75E+01 & 1140.9677 \\
58081 & 2 & 144480 & 2 &   6.02E-02 &   4.84E-06 &   3.09E-03 & 15.6060 \\
58081 & 3 & 303120 & 2 &   4.41E-03 &   3.72E-06 &   2.74E-07 & 1.1430 \\
\hline
\end{tabular}
\end{center}
\end{table}
It is easy to see that lowest order Raviart-Thomas elements are not able to present
the flux properly and in the case $p_2=1$, the efficiency index of the majorant is poor.
Again, in the $p$-refined spaces the estimate improves significantly.

\begin{Example} \label{ex:2}
Let $\Omega:=(0,1)\times(0,1)\times(0,1)$, $f(x_1,x_2,x_3)=x_1 x_2 x_3$, \linebreak and
\[
\mathbf{A} = \left( \begin{smallmatrix} 1000 & 20 & -500 \\ -3 & 30 & 16 \\ 2 0 3 \end{smallmatrix} \right) .
\]
Then,
\[
\mathbf{A}^{-1} \approx \left( \begin{smallmatrix}
 7.4490978E-04 & -4.9660652E-04 &  1.2680020E-01 \\
 3.3934779E-04 &  3.3107104E-02 & -1.2001324E-01 \\
-4.9660652E-04 &  3.3107101E-04 &  2.4879987E-01
 \end{smallmatrix} \right)
\]
and
\[
\mathbf{B} \approx \left( \begin{smallmatrix}
 1.0126139 & -0.4980245 &  2.0416897 \\
-0.0160603 & 0.5154516  & -0.0408795 \\
-0.0060666 & 0.009230   & -0.0280656
 \end{smallmatrix} \right)
\] .
\end{Example}
In Example \ref{ex:2}, the exact solution is not known. Instead a reference solution was
computed
using third order Courant type elements with 29791 global degrees of freedom is applied.
The approximations were computed using linear tetrahedral Courant type elements and the
fluxes are generated using tetrahedral Raviart-Thomas elements of order $p_2$. The
results were depicted on Table \ref{tab:ex3} and they show similar characteristics as in
the two dimensional example.
\begin{table}
\begin{center}
\caption{Example \ref{ex:2}, $p_1=1$}
\label{tab:ex3}
\begin{tabular}{@{}lllllllllll@{}}
\hline
$N_1$ & $p_2$ & $N_2$ & $k$ & $\maj^2(v,\mathbf{y}_k,\beta_k)$
& $\maj^{\rm Dual}_k$ & $\maj^{\rm Equi}_k$ & $I_{\rm eff}$ \\
\hline
125 & 1 & 864 & 4 &   4.67E-02 &   1.15E-05 &   1.57E-03 & 10.0122 \\
125 & 2 & 3744 & 3 &   8.47E-03 &   6.99E-06 &   9.43E-06 & 1.8164 \\
125 & 3 & 9792 & 3 &   5.26E-03 &   6.68E-06 &   8.94E-09 & 1.1284 \\
343 & 1 & 2808 & 3 &   3.12E-02 &   5.81E-06 &   6.85E-04 & 9.2258 \\
343 & 2 & 12312 & 3 &   5.24E-03 &   3.65E-06 &   1.86E-06 & 1.5489 \\
343 & 3 & 32400 & 3 &   3.81E-03 &   3.65E-06 &   7.85E-10 & 1.1241 \\
729 & 1 & 6528 & 3 &   2.35E-02 &   3.48E-06 &   3.83E-04 & 9.1082 \\
729 & 2 & 28800 & 3 &   3.78E-03 &   2.21E-06 &   5.88E-07 & 1.4642 \\
729 & 3 & 76032 & 3 &   2.97E-03 &   2.64E-06 &   1.40E-10 & 1.1527 \\
1331 & 1 & 12600 & 3 &   1.88E-02 &   2.31E-06 &   2.44E-04 & 7.8120 \\
1331 & 2 & 55800 & 3 &   2.94E-03 &   1.47E-06 &   2.41E-07 & 1.2208 \\
\hline
\end{tabular}
\end{center}
\end{table}

\section{Summary}

An upper functional deviation estimate (majorant) for nonsymmetric stationary diffusion problem is derived.
An algorithm for the global minimization of the majorant over a finite dimensional subspace
is presented and tested.
The efficiency of the majorant depends on the particular problem (i.e., the exact solution)
and the relation of spaces $V_h$ and $Q_h$. The question is that how accurately $V_h$ can
represent $u$ (in the energy norm) in comparison with the ability of $Q_h$ to represent
$A\nabla u$ (in the $\Hdiv$-norm). If $Q_h$ is
``better'', then the estimate is very accurate and the other way round. The crude overestimation
in Table \ref{tab:ex2} shows that using a ``worse'' space for the computation of fluxes is
not generally a good idea.

\pagebreak


\begin{thebibliography}{1}

\bibitem{Denisenko1989}
V.~V. Denisenko.
\newblock Variational methods for elliptic boundary value problems that
  describe transport processes with nonsymmetric tensor coefficients.
\newblock {\em Zh. Prikl. Mekh. i Tekhn. Fiz.}, (3):69--75, 1989.

\bibitem{Denisenko1997}
V.~V. Denisenko.
\newblock The energy method for three-dimensional elliptic equations with
  nonsymmetric tensor coefficients.
\newblock {\em Sibirsk. Mat. Zh.}, 38(6):1267--1281, ii, 1997.

\bibitem{Hecht2012}
F.~Hecht.
\newblock New development in freefem++.
\newblock {\em J. Numer. Math.}, 20(3-4):251--265, 2012.

\bibitem{LoggMardalEtAl2012a}
Anders Logg, Kent-Andre Mardal, Garth~N. Wells, et~al.
\newblock {\em Automated Solution of Differential Equations by the Finite
  Element Method}.
\newblock Springer, 2012.

\bibitem{MaliNeittaanmakiRepin2014}
O.~Mali, S.~Repin, and P.~Neittaanm{\"a}ki.
\newblock {\em Accuracy verification methods, theory and algorithms}, volume~32
  of {\em Computational Methods in Applied Sciences}.
\newblock Springer, 2014.

\bibitem{NeittaanmakiRepin2004}
P.~Neittaanm{\"a}ki and S.~Repin.
\newblock {\em Reliable methods for computer simulation, Error control and a
  posteriori estimates}.
\newblock Elsevier, New York, 2004.

\bibitem{RaviartThomas1977}
P.~A. Raviart and J.~M. Thomas.
\newblock Primal hybrid finite element methods for 2nd order elliptic
  equations.
\newblock {\em Math. Comp.}, 31(138):391--413, 1977.

\bibitem{Repin1997}
S.~Repin.
\newblock A posteriori estimates for approximate solutions of variational
  problems with strongly convex functionals.
\newblock {\em Problems of Mathematical Analysis}, 17:199--226, 1997.

\bibitem{Repin2008}
Sergey Repin.
\newblock {\em A posteriori estimates for partial differential equations},
  volume~4 of {\em Radon Series on Computational and Applied Mathematics}.
\newblock Walter de Gruyter GmbH \& Co. KG, Berlin, 2008.

\end{thebibliography}

\def\cprime{$'$}

\end{document}